\documentclass[a4paper, 11pt]{article}
\usepackage[mathscr]{eucal}
\usepackage{amssymb}
\usepackage{latexsym}
\usepackage{amsthm}
\usepackage{amsmath}
\usepackage[dvips]{graphicx}
\usepackage{psfrag}
\usepackage{a4wide}

\newtheorem{theorem}{Theorem}[section]
\newtheorem{proposition}[theorem]{Proposition}
\newtheorem{lemma}[theorem]{Lemma}

\theoremstyle{definition}

\newtheorem{problem}[theorem]{Problem}

\makeatletter
\@addtoreset{equation}{section}

\makeatother

\title{On the largest subsets avoiding the diameter of $(0,\pm 1)$-vectors} 

\author{
Saori Adachi and Hiroshi Nozaki
}
\begin{document}
\maketitle

\renewcommand{\thefootnote}{\fnsymbol{footnote}}
\footnote[0]{2010 Mathematics Subject Classification: 
05D05  
(05C69).
\\
\noindent
{\it Saori Adachi}: 
	Mathematics Education, 
	Graduate School of Education, 
	Aichi University of Education, 
	1 Hirosawa, Igaya-cho, 
	Kariya, Aichi 448-8542, 
	Japan.
	s214m044@auecc.aichi-edu.ac.jp.
\\ \noindent
{\it Hiroshi Nozaki}: 
	Department of Mathematics Education, 
	Aichi University of Education, 
	1 Hirosawa, Igaya-cho, 
	Kariya, Aichi 448-8542, 
	Japan.
	hnozaki@auecc.aichi-edu.ac.jp.
}

\begin{abstract}
Let $L_{mkl}\subset \mathbb{R}^{m+k+l}$ be 
the set of vectors which have $m$ of entries $-1$, $k$ of entries $0$, and $l$ of entries $1$. 
In this paper, we investigate 
the largest subset of $L_{mkl}$ 
whose diameter is smaller than that of 
$L_{mkl}$. The largest subsets for 
$m=1$, $l=2$, and any $k$ will be classified. From this result, we can classify the largest $4$-distance sets containing the Euclidean representation of the Johnson scheme $J(9,4)$. 
This was an open problem in Bannai, Sato, and Shigezumi (2012). 
\end{abstract}
\textbf{Key words}:
the Erd\H{o}s--Ko--Rado theorem, 
$s$-distance set,
diameter graph,  
independent set, 
extremal set theory. 

\section{Introduction}
The famous theorem in Erd\H{o}s--Ko--Rado \cite{EKR61} stated that for $n \geq 2k$ and a family $\mathfrak{A}$ of $k$-element subsets of $I_n=\{1,\ldots, n\}$, if any two distinct $A,B \in \mathfrak{A}$ satisfy $A\cap B \ne \emptyset$, then
\[
|\mathfrak{A}| \leq \binom{n-1}{k-1}. 
\]
For $n>2k$, the set  
$
\{A\subset I_n \mid |A|=k, 1 \in A \}
$
is the unique family achieving equality, up to permutations on $I_n$. 
For $n=2k$, the largest set is any family  
which contains only one of $A$ or $I_n\setminus A$ for any $k$-element $A\subset I_n$. 
This result plays a central role in extremal set theory, and similar or analogous theorems are proved for various objects \cite{B12,DF83,F87}. 

We can naturally interpret  $A \subset I_n$ as $x=(x_1, \ldots, x_n)\in \mathbb{R}^n$ by the manner 
$x_i=1$ if $i \in A$, $x_i=0$ if $i \not\in A$.  
By this identification, 
the Erd\H{o}s--Ko--Rado theorem can be rewritten that for
$n \geq 2k$ and a 
subset $X$ of $L_k=\{x \in \mathbb{R}^n \mid x_i \in \{0,1\}, \sum x_i=k \}$ if any
distinct $x,y \in X$ satisfy $d(x,y) < D(L_k)=\sqrt{2k} $, then 
\[
|X| \leq \binom{n-1}{k-1},
\]   
where $d(,)$ is the Euclidean distance, and $D(L_k)$ is the diameter of $L_k$. 
We would like to consider the following problem to generalize the Erd\H{o}s--Ko--Rado theorem. 
\begin{problem}
Let $L_{mkl}\subset \mathbb{R}^{m+k+l}$ be 
the set of vectors which have $m$ of entries $-1$, $k$ of entries $0$, and $l$ of entries $1$. 
Classify the largest $X \subset L_{mkl}$ with $D(X)<D(L_{mkl})$. 
\end{problem}
It is almost obvious for the cases $m=l$ (Proposition~\ref{prop:m=l}) and 
$m+k \leq l$ (Proposition~\ref{prop:m+k<=l}). In this paper, we solve the first non-trivial case $m=1$, $l=2$ and any $k$ (Theorem~\ref{thm:main}). 
Using the largest sets for the case 
$(m,k,l)=(1,6,2)$, we can classify 
the largest $4$-distance sets containing the Euclidean representation of 
the Johnson scheme $J(9,4)$.  
This was an open problem in 
\cite{BSS12}. 

We will give a brief survey on related results. 
Let $\mathfrak{L}_{nm}$ be the set of $(0,\pm 1)$-vectors in $\mathbb{R}^n$ which have $m$ non-zero coordinates. 
For a fixed set $D$ of integers, let $V(n,m,D)$ be the family of subsets $V=\{v_1,\ldots,v_k \}$ of $\mathfrak{L}_{nm}$ such that $(v_i,v_j)\in D$ for any $i\ne j$. 
There are several results relating to the largest sets in $V(n,m,D)$ for some $(n,m,D)$ \cite{DF81,DF83_2,DF85}. 
Since $X \subset \mathfrak{L}_{nm}$ is on a sphere, 
if $|D|=s$ holds, then $|X| \leq \binom{n+s-1}{s}+\binom{d+s-2}{s-1}$ \cite{DGS77}.  
The case $D=\{d \}$ is investigated in \cite{DF81}. 
For non-negative integers $d<m$, $t\geq 2$, and $n>n_0(m)$ (see \cite{DF81} about $n_0(m)$), if $X \in 
V(n,m,\{d,d+1,\ldots,d+t-1\})$, then $|X|\leq \binom{n-d}{t}/\binom{m-d}{t}$~\cite{DF83_2}. 
This equality can be attained whenever a Steiner system $S(n-d,m-d,t)$ (equivalently $t$-$(n-d,m-d,1)$ design) exists . 
We also have  
if $X \in 
V(n,m,\{-(t-1),-(t-2),\ldots, t-1\})$, then $|X|\leq 2^{t-1}(m-t+1)\binom{n}{t}/\binom{m}{t}$ \cite{DF85}.
When $m=t+1$, this equality can be attained whenever a Steiner system $S(n,m,m-1)$ exists. 
  
%

\section{Largest subsets avoiding the diameter of $L_{mkl}$} \label{sec:2}
Let $L_{mkl}$ denote the finite set in $\mathbb{R}^n=\mathbb{R}^{m+k+l}$, which consists of all vectors whose number of entries $-1$, $0$, $1$ is equal to $m$, $k$, $l$, respectively. 
For two subsets $X,Y$ of $L_{mkl}$,  
$X$ is {\it isomorphic} to $Y$ if there exists 
a permutation $\sigma \in S_{n}$ such that 
$X=\{(y_{\sigma(1)},\ldots,y_{\sigma(n)}) \mid (y_1,\ldots,y_n) \in Y\}$. 
The {\it diameter} $D(X)$ of $X\subset \mathbb{R}^n$ is defined to be  
\[
D(X)= \max \{d(x,y) \mid x, y  \in X \}, 
\]
where $d(,)$ is the Euclidean distance. 
Let $M_{mkl}$ denote the largest possible number of cardinalities of $X \subset L_{mkl}$ such that 
$D(X)<D(L_{mkl})$. 
The {\it diameter graph} of $X\subset \mathbb{R}^n$ is defined to be the graph $(X,E)$, where 
$E=\{(x,y) \mid d(x,y)=D(X)\}$. 
The problem of determining $M_{mkl}$ is equivalent to determining the independence number of the diameter graph of $L_{mkl}$. 
 Note that $M_{mkl}=M_{lkm}$ because we have $L_{mkl}=-L_{lkm}=\{-x \mid  x\in L_{lkm}\}$. Thus we may assume $m \leq l$. In this section, we determine $M_{mkl}$, and classify the largest sets
for several cases of $m$, $k$, $l$. 
 
First we determine $M_{mkl}$ for the  cases $m=l$ and $m+k\leq l$. 
\begin{proposition}\label{prop:m=l}
Assume $m=l$. Then we have
\[
M_{mkl} = \frac{1}{2}\binom{n}{m}\binom{k+m}{m}=\frac{1}{2}|L_{mkl}|, 
\]
and the largest sets contain only one of $x$ or $-x$ for any $x \in L_{mkl}$.   
\end{proposition}
\begin{proof}
For any $x \in L_{mkl}$, we have 
$\{y \mid d(x,y)=D(L_{mkl})\}=\{-x\}$. Therefore the diameter graph of $L_{mkl}$ is the set of independent edges. The proposition can be easily proved from this fact. 
\end{proof}
For $X \subset L_{mkl}$, we use the notation 
\[
N_i(X,j)=\{(x_1,\ldots,x_n) \in X \mid x_i=j\}, \qquad \text{ and } \qquad n_i(X,j) = |N_i(X,j)|. 
\]
\begin{proposition} \label{prop:m+k<=l}
Assume $m+k \leq l$. Then we have 
\[
M_{mkl} = \binom{n-1}{m+k-1}\binom{m+k}{m}.
\]
For $m+k>l$, the largest set is $N_1(L_{mkl},-1)\cup N_1(L_{mkl},0)$, up to isomorphism. 
For $m+k=l$, then the largest sets  contain only one of 
$\{(x_1,\ldots,x_n) \in L_{mkl} \mid  x_i=1, \forall i \in J \}$ or $\{(x_1,\ldots,x_n) \in L_{mkl} \mid  x_i=1,\forall i \in  I_n \setminus J \}$ for any $J \subset I_n$ of order $l$. 
\end{proposition}
\begin{proof}
A finite subset $X$ of $L_{mkl}$ satisfies 
$D(X)<D(L_{mkl})$ if and only if $\{i \mid x_i=-1,0\} \cup \{i \mid y_i=-1,0\}$ is not empty for any distinct $(x_1,\ldots,x_n),(y_1,\ldots, y_n) \in X$. 
We can therefore apply the Erd\H{o}s--Ko--Rado theorem \cite{EKR61}
to determine the positions of entries $-1$ or $0$.  The number of possible positions of $-1$, $0$ is $\binom{n-1}{m+k-1}$. After fixing the position,  $-1$, $0$ can be placed in $\binom{m+k}{k}$ ways.  
 This determines $M_{mkl}$.  The largest sets are classified from the optimal sets of the Erd\H{o}s--Ko--Rado theorem. 
\end{proof}
The remaining part of this section is devoted to 
proving
\[ 
M_{1k2} = \mathfrak{M}_k=\binom{k+3}{3}+2,
\] 
and determining the classification of the largest sets. Note that $D(L_{1k2})=\sqrt{10}$ and if $X\subset L_{1k2}$ satisfies $D(X)<D(L_{1k2})$, then $D(X) \leq \sqrt{8}$. 
The following two lemmas are used later. 
\begin{lemma} \label{lem:M_k}
Let $X\subset L_{1k2}$ with $D(X)< D(L_{1k2})$. 
Suppose $k\geq 4$, and $|X| \geq \mathfrak{M}_k$. 
Then there exists $i \in \{1,\ldots,n\}$ such that
$n_i(X,0)\geq \mathfrak{M}_{k-1}$. 
\end{lemma}
\begin{proof}
This lemma is immediate because the average of 
$n_i(X,0)$ is 
\[
\frac{1}{n}\sum_{i=1}^{n}n_{i}(X,0)=\frac{k|X|}{k+3}
\geq \frac{k\mathfrak{M}_k}{k+3}
=\mathfrak{M}_{k-1}-\frac{6}{k+3}
>\mathfrak{M}_{k-1}-1. \qedhere
\]
\end{proof}
\begin{lemma} \label{lem:match}
Let $G=(V,E)$ be a connected simple graph, 
and $E'$ a matching in $G$. 
Assume that $G$ has an independent set $I$ of size 
$|V|-|E'|$.   
Then for $z\in I$ if $x \in V$ satisfies  
 $(x,y)\in E'$ for some $y$ adjacent to $z$, 
then $x \in I$. 
\end{lemma}
\begin{proof}
Since the cardinality of $I$ is $|V|-|E'|$, only one of $x$ or $y$ is an element of $I$ for any $(x,y) \in E'$.
By assumption, $y \not \in I$, and hence $x \in I$.  
\end{proof}
The subsets $S_k(i)$, $T_k(i)$, $U_k(i)$ of $L_{1k2}$ are defined by 
\begin{align*}
S_k(i)&=\{(x_1,\ldots,x_n) \in L_{1k2}\mid x_1=\cdots=x_{i-1}=0, x_i=-1\}, \\
T_k(i)&=\{(x_1,\ldots,x_n) \in L_{1k2}\mid x_1=\cdots=x_{i-1}=0, x_i=1\}, \\
U_k(i)&=\{(x_1,\ldots,x_n) \in L_{1k2} \mid 
x_1=1, x_l=-1, x_j=1, 
\exists l
\in \{2,\ldots, i\},  
\exists j \in \{l+1,\ldots,n\}\} 
\end{align*}
for $i=2,\ldots,k+2$. 
We define $S_k(1)=N_1(L_{1k2},-1)$, and 
$T_k(1)=N_1(L_{1k2},1)$. 
The following are candidates of the largest subsets avoiding the largest distance $\sqrt{10}$. 
\begin{align*}
&X_k=T_k(k+1)\cup (\bigcup_{i=1}^{k+1} S_k(i)) \text{ for } k\geq 1, \\
&Y_1=T_1(1), \qquad Y_k=T_k(k)\cup (\bigcup_{i=1}^{k-1} S_k(i)) \text{ for } k\geq 2, \\
&Z_2=T_2(1), \qquad Z_k=T_k(k-1)\cup( \bigcup_{i=1}^{k-2} S_k(i)) \text{ for } k\geq 3.
\end{align*}
Note that $|X_k|=|Y_k|=|Z_k|=\mathfrak{M}_k$, and they can be inductively constructed by 
\begin{align*}
X_k&=\{(0,x)\mid x \in X_{k-1}\} \cup N_1(L_{1k2},-1), \\ 
Y_k&=\{(0,x)\mid x \in Y_{k-1}\} \cup N_1(L_{1k2},-1), \\ 
Z_k&=\{(0,x)\mid x \in Z_{k-1}\} \cup N_1(L_{1k2},-1). 
\end{align*}
We also use the following notation.
\begin{align*}
X_k'&=X_k\setminus S_k(1)\quad (k\geq 2)& \qquad 
Y_k'&=Y_k\setminus S_k(1)\quad (k\geq 2)& \qquad 
Z_k'&=Z_k\setminus S_k(1)\quad (k\geq 3). \\
 &=\{(0,x)\mid x \in X_{k-1}\}, & &=\{(0,x)\mid x \in Y_{k-1}\}, & &=\{(0,x)\mid x \in Z_{k-1}\}.
\end{align*}
\begin{theorem} \label{thm:main}
Let $X\subset L_{1k2}$ with $D(X)<D(L_{1k2})$. Then we have
\[
|X| \leq \mathfrak{M}_k. 
\]
If equality holds, then 
\begin{enumerate}
\item for $k=1$, $X=X_1$, or $Y_1$,
\item for $k\geq 2$, $X=X_k$, $Y_k$, or $Z_k$,
\end{enumerate}
up to isomorphism. 
\end{theorem}
This theorem will be proved by induction. 
 We first prove the inductive step. 
\begin{lemma}\label{lem:induc}
Let $k\geq 2$.  Assume that the statement in Theorem~\ref{thm:main}
holds for some $k-1$. Let $X \subset L_{1k2}$ with $D(X)<D(L_{1k2})$, such that $n_i(X,0)=\mathfrak{M}_{k-1}$ for some $i$.
Then we have $|X| \leq \mathfrak{M}_{k}$. 
If equality holds, then  $X=X_{k}$, $Y_{k}$, or $Z_{k}$, up to isomorphism. 
\end{lemma}
\begin{proof}
Without loss of generality, 
$n_1(X,0)=\mathfrak{M}_{k-1}$, and hence 
$X$ contains $X_k'$, $Y_k'$, or $Z_k'$ for $k\geq 3$, and $X_1'$, or $Y_1'$ for $k=2$.

(i) Suppose $X_k' \subset X$ for $k\geq 2$. 
The set of other candidates of elements of $X$
is $S_k(1)\cup U_k(k)$. 
The diameter graph $G$ of $S_k(1)\cup U_k(k)$
is a bipartite graph of the partite sets $S_k(1)$ and $U_k(k)$. Since the three elements 
\[(-1,0,\ldots,0,0,1,1),(-1,0,\ldots,0,1,0,1),
(-1,0,\ldots,0,1,1,0)\in S_k(1)\]
are isolated vertices in $G$, they may be contained in $X$. Let $G'$ be the 
subgraph of $G$ formed by removing the three isolated vertices. A perfect matching of $G'$ is given as follows.
\begin{center}
Matching (i)\\
\begin{tabular}{|c|c|}
\hline
$S_k(1)$ &$U_k(k)$ \\ 
$(-1,x_2, \ldots, x_{k+3})$ &  $(1,y_2,\ldots,y_{k+3})$ \\ \hline
$x_i=1,x_j=1$ ($2\leq i\leq k, i<j<n $)& $y_i=-1,y_{j+1}=1 $ \\
$x_i=1,x_n=1$ ($2\leq i\leq k $)& $y_i=-1,y_{i+1}=1 $ \\
\hline
\end{tabular}
\end{center}
By this matching, we can show
\[
|X| \leq \mathfrak{M}_{k-1}+|S_k(1)|=\mathfrak{M}_k.
\]

We will classify the sets attaining this bound. 
First assume that $x \in X$ for some $x\in S_k(1)$ with $x_2=1$. 
By Lemma~\ref{lem:match}, $X$ must contain any $x \in S_k(1)$ with $x_2=1$. In particular, $(-1,1,1,0,\ldots,0) \in X$. Using Lemma~\ref{lem:match} again, $X$ must contain $x \in S_k(1)$ with $x_3=1$. By a similar manner, $X$ must contain any $x \in S_k(1)$. Therefore $X=X_k$. 

Assume $X$ does not contain any $x\in S_k(1)$ with $x_2=1$, namely $n_2(X,1)=0$. 
By assumption, we have 
\[
|X|=n_2(X,-1)+n_2(X,0)\leq \binom{k+2}{2}+\mathfrak{M}_{k-1}=\mathfrak{M}_k. 
\]
If $|X|=\mathfrak{M}_k$, then we have 
$n_2(X,-1)=\binom{k+2}{2}$ and $n_2(X,0)=\mathfrak{M}_{k-1}$. This implies that $X$ is isomorphic to 
$X_k$, $Y_k$, or $Z_k$. 

(ii) Suppose $Y_k' \subset X$ for $k\geq 2$. 
The set of other candidates of elements of $X$
is the union of $S_k(1)$, $U_k(k-1)$, and 
\[
\mathcal{S}_1=\{(x_1,\ldots, x_{k+3}) \in 
L_{1k2} \mid x_1=1, x_k=1,x_j=-1, k< j \}
\]
for $k\geq 3$, 
and $S_2(1)\cup \mathcal{S}_1$ for
$k=2$.  
The diameter graph $G$ of $S_k(1)\cup U_k(k-1)\cup \mathcal{S}_1$
is a bipartite graph of the partite sets $S_k(1)$ and $U_k(k-1)\cup \mathcal{S}_1$. Since the three elements 
\[(-1,0,\ldots,0,1,1,0,0),(-1,0,\ldots,0,1,0,1,0),
(-1,0,\ldots,0,1,0,0,1)\in S_k(1)\]
are isolated vertices in $G$, they may be contained in $X$. Let $G'$ be the 
subgraph of $G$ formed by removing the three isolated vertices. A perfect matching of $G'$ is given as follows.
\begin{center}
Matching (ii)\\
\begin{tabular}{|c|c|}
\hline
$S_k(1)$ &$U_k(k-1)$ \\ 
$(-1,x_2, \ldots, x_{k+3})$ &  $(1,y_2,\ldots,y_{k+3})$ \\ \hline
$x_i=1,x_j=1$ ($2\leq i\leq k-1, i<j<n $)& $y_i=-1,y_{j+1}=1 $ \\
$x_i=1,x_n=1  $ ($2\leq i\leq k-1 $)& $y_i=-1,y_{i+1}=1 $ \\
\hline
\end{tabular}

\begin{tabular}{|c|c|}
\hline
$S_k(1)$ &$\mathcal{S}_1$ \\ 
 \hline
$(-1,0,\ldots,0,1,1,0)$& $(1,0,\ldots,0,1,-1,0,0)$ \\
$(-1,0,\ldots,0,0,1,1)$& $(1,0,\ldots,0,1,0,-1,0)$ \\
$(-1,0,\ldots,0,1,0,1)$& $(1,0,\ldots,0,1,0,0,-1)$ \\
\hline
\end{tabular}
\end{center}
By this maching, we can show
$
|X| \leq \mathfrak{M}_k.
$

We will classify the sets attaining this bound. 
For $k=2$, the maximum indepdent sets of 
$G'$ is $\{(-1,0,0,1,1),(-1,0,1,0,1),(-1,0,1,1,0)\} \subset S_2(1)$ or $\mathcal{S}_1$. This implies that $X=Y_2$ or $Z_2$. 
For $k\geq 3$, we assume that $x \in X$ for some $x\in S_k(1)$ with $x_2=1$. 
By Lemma~\ref{lem:match}, $X$ must contain any $x \in S_k(1)$. Therefore $X=Y_k$. 
If $X$ does not contain any $x\in S_k(1)$ with $x_2=1$, namely $n_2(X,1)=0$. 
It can be proved that $X$ is isomorphic to 
$X_k$, $Y_k$, or $Z_k$. 

(iii) Suppose $k\geq 3$, and $Z_k' \subset X$. 
The set of other candidates of elements of $X$
is the union of $S_k(1)$, $U_k(k-2)$, and 
\[
\mathcal{S}_2=\{(x_1,\ldots, x_{k+3}) \in L_{1k2} \mid x_1=1, x_{k-1}=1,x_j=-1, k< j \}
\]
for $k\geq 4$, 
and $S_3(1)\cup \mathcal{S}_2$ for
$k=3$.  
The diameter graph $G$ of $S_k(1)\cup U_k(k-2)\cup \mathcal{S}_2$
is a bipartite graph of the partite sets $S_k(1)$ and $U_k(k-2)\cup \mathcal{S}_2$. Since the four vectors 
\begin{align*}
&(-1,0,\ldots,0,1,1,0,0,0),(-1,0,\ldots,0,1,0,1,0,0),\\
&(-1,0,\ldots,0,1,0,0,1,0),(-1,0,\ldots,0,1,0,0,0,1)\in S_k(1)
\end{align*}
are isolated vertices in $G$, they may be contained in $X$. Let $G'$ be the 
subgraph of $G$ formed by removing the four isolated vertices. A maximum matching of $G'$ is given as follows.
\begin{center}
Matching (iii)\\
\begin{tabular}{|c|c|}
\hline
$S_k(1)$ &$U_k(k-2)$ \\ 
$(-1,x_2, \ldots, x_{k+3})$ &  $(1,y_2,\ldots,y_{k+3})$ \\ \hline
$x_i=1,x_j=1$ ($2\leq i\leq k-2, i<j<n $)& $y_i=-1,y_{j+1}=1 $ \\
$x_i=1,x_n=1  $ ($2\leq i\leq k-2 $)& $y_i=-1,y_{i+1}=1 $ \\
\hline
\end{tabular}

\begin{tabular}{|c|c|}
\hline
$S_k(1)$ &$\mathcal{S}_2$ \\
 \hline 
$(-1,0,\ldots,0,1,1,0,0)$& $(1,0,\ldots,0,1,-1,0,0,0)$ \\
$(-1,0,\ldots,0,0,1,1,0)$& $(1,0,\ldots,0,1,0,-1,0,0)$ \\
$(-1,0,\ldots,0,0,0,1,1)$& $(1,0,\ldots,0,1,0,0,-1,0)$ \\
$(-1,0,\ldots,0,1,0,0,1)$& $(1,0,\ldots,0,1,0,0,0,-1)$ \\
\hline
\end{tabular}
\end{center}
Note that the two vectors
 \begin{equation} \label{eq:1}
(-1,0,\ldots,0,1,0,1,0),(-1,0,\ldots,0,0,1,0,1)\in S_k(1)
\end{equation}
are unmatched in this matching. 
By this matching, we can show
$
|X| \leq \mathfrak{M}_k.
$

We will classify the sets attaining this bound. 
If $|X|=\mathfrak{M}_{k}$, then the two vectors in \eqref{eq:1} must be contained in 
$X$. Therefore $X$ does not contain any element of $\mathcal{S}_2$, and contains 
an element of $S_k(1)$ which matches some element of 
$\mathcal{S}_2$. 
For $k=3$, $X$ therefore contains $S_k(1)$, and $X=Z_3$. 
For $k\geq 4$, we assume that $x \in X$ for some $x\in S_k(1)$ with $x_2=1$. 
By Lemma~\ref{lem:match}, $X$ must contain any $x \in S_k(1)$. Therefore $X=Z_k$. 
If $X$ does not contain any $x\in S_k(1)$ with $x_2=1$, namely $n_2(X,1)=0$. 
Therefore $X$ is isomorphic to 
$X_k$, $Y_k$, or $Z_k$. 
\end{proof} 
Matchings (i)--(iii) and the notation $\mathcal{S}_1$, $\mathcal{S}_2$ defined in the proof of Lemma~\ref{lem:induc} are used again later. 
The base case in the induction is the case $k=3$. We will prove the cases $k=1,2,3$ in order.
\begin{proposition}\label{prop:k1}
Let $X\subset L_{112}$ with $D(X)<D(L_{112})$. Then we have
\[
|X| \leq \mathfrak{M}_1=6. 
\]
If equality holds, then $X=X_1$, or $Y_1$, up to isomorphism. 
\end{proposition}
\begin{proof}
Since the diameter graph $G$ of $L_{112}$ is isomorphic to $C_4\cup C_4 \cup C_4$, where $C_4$ is the $4$-cycle, the bound $|X|\leq 6$ clearly holds. 
Considering the permutation of coordinates,  $G$ has the automorphism group $S_4$. 
Since the stabilizer of $X_1$ in $S_4$ is of order $6$, the orbit of $X_4$ has length $4$. 
Similarly the orbit of $Y_1$ has length $4$. Since the number of maximum independent sets of $G$ is $2^3=8$, 
this proposition follows. 
\end{proof}
For $k=2$, we also classify $(\mathfrak{M}_2-1)$-point sets $X$ with $D(X)<D(L_{122})$ in order to prove the case $k=3$. 
\begin{proposition}\label{prop:k2}
Let $X\subset L_{122}$ with $D(X)<D(L_{122})$. Then we have
\[
|X| \leq \mathfrak{M}_2=12. 
\]
If $|X|=12$, then $X=X_2$, $Y_2$, or $Z_2$, up to isomorphism. 
If $|X|=11$, then $X$ is  
\begin{equation*} \label{eq:pro1}
V_2=X_2' \cup \{(-1,0,0,1,1),(-1,0,1,0,1),(-1,0,1,1,0),(-1,1,1,0,0),(1,-1,1,0,0) \},
\end{equation*}
\begin{equation*} \label{eq:pro2}
W_2=Y_2'\cup \{(-1,1,1,0,0),(-1,1,0,1,0),(-1,1,0,0,1),(-1,0,0,1,1),(1,1,-1,0,0) \},
\end{equation*}
or the set 
obtained by removing a point from $X_2$, $Y_2$, or $Z_2$,  
up to isomorphism. 
\end{proposition}
\begin{proof}
First suppose $n_i(X,0)=6$ for some $i$.  Then we have $|X|\leq 12$, and $X$ with $|X|=12$ is $X_2$, $Y_2$, or $Z_2$ by Lemma~\ref{lem:induc}. 
In order to find $X$ with $|X|=11$, 
we consider $5$-point independent sets in 
the diameter graph of $S_2(1) \cup U_2(2)$ or 
$S_2(1) \cup U_2(1) \cup \mathcal{S}_1$. 
If $X$ is not isomorphic to a subset of $X_2$, $Y_2$, or $Z_2$, 
then $X=V_2$ from $S_2(1) \cup U_2(2)$, and 
$X=W_2$ from $S_2(1)\cup U_2(1) \cup \mathcal{S}_1$. 

Suppose $n_i(X,0)\leq 5$ for any $i$. 
If $|X|\geq 11$, then 
the average of $n_i(X,0)$ is greater than $4$. Without loss of generality, we may assume $n_1(X,0)=5$. 
Since the diameter graph of 
$L_{112}$ is $C_4\cup C_4 \cup C_4$, we can show that $X$ contains a $5$-point subset of $X_2'$ or $Y_2'$. 

(i) Suppose $X$ contains a $5$-point subset of $X_2'$. By considering the automorphism group of $X_2'$, we may assume 
$X$ contains the $5$-point subset obtained by removing $(0,-1,0,1,1)$ or $(0,0,-1,1,1)$. First assume that $X$ contains the $5$-point subset 
obtained by removing $(0,-1,0,1,1)$. Since other candidates of elements of $X$ are still in $S_2(1)\cup U_2(2)$, we have
$|X|\leq 11$, and if $|X|=11$, then
$X$ is isomorphic to a subset of $X_2$, $Y_2$, or $Z_2$. Assume that $X$ contains the $5$-point subset 
obtained by 
removing $(0,0,-1,1,1)$.
The set of other candidates of elements of $X$ is 
$S_2(1)\cup U_2(2)\cup 
\{(1,0,1,-1,0),(1,0,1,0,-1) \}$. 
If $X$ does not contain both 
$(1,0,1,-1,0)$ and $(1,0,1,0,-1)$, then $|X| \leq 11$, and  
$X$ attaining this bound is isomorphic to a subset of $X_2$, $Y_2$, or $Z_2$. 
To make a new set, $X$ may contain $(1,0,1,-1,0)$. The two vectors $(-1,1,0,1,0), 
(-1,0,0,1,1) \in S_2(1)$, which are at distance $\sqrt{10}$ from $(1,0,1,-1,0)$, are not contained in $X$. The set $P_1$ consisting of the two isolated vertices
\[
(-1,0,1,0,1),(-1,0,1,1,0) \in S_2(1)
\]
and $6$ points  
\[
(-1,1,1,0,0),(-1,1,0,0,1),(1,-1,1,0,0),
(1,-1,0,1,0),
 (1,-1,0,0,1),(1,0,1,0,-1) 
\]
has the unique maximum $6$-point independent set
\[
\{(-1,0,1,0,1),(-1,0,1,1,0), (1,-1,1,0,0),
(1,-1,0,1,0),
 (1,-1,0,0,1),(1,0,1,-1,0)\}, 
\]
which gives $X$ isomorphic to $Y_2$, and 
$n_2(X,0)=6$.  
If $X$ contains a $5$-point independent set 
in $P_1$ and  is not isomorphic to a subset of  $Y_2$,  then $X$  contains the $5$-point independent set
\[
\{(-1,0,1,0,1),(-1,0,1,1,0), (-1,1,1,0,0),(1,-1,1,0,0),(1,0,1,0,-1)\}. 
\]
Then $X$ is isomorphic to $W_2$ and
 $n_2(X,0)=6$. 

(ii) Suppose $X$ contains a $5$-point subset of $Y_2'$. By considering the automorphism group of $Y_2'$, we may assume 
$X$ contains the $5$-point subset 
obtained by 
removing $(0,1,-1,0,1)$. 
The set of other candidates of elements of $X$ is $S_2(1)\cup \mathcal{S}_1\cup \{(1,0,1,0,-1)\}$. 
To make a new set, $X$ may contain $(1,0,1,0,-1)$. The two vectors $(-1,1,0,0,1),(-1,0,0,1,1) \in S_2(1)$, which are at distance $\sqrt{10}$ from $(1,0,1,0,-1)$, are not contained in $X$. The set consisting of the two isolated vertices
\[
(-1,1,1,0,0),(-1,1,0,1,0) \in S_2(1)
\]
and $5$ points  
\[
(-1,0,1,1,0),
(-1,0,1,0,1),  
(1,1,-1,0,0),
(1,1,0,-1,0),
(1,1,0,0,-1)
\]
has 
the unique maximum $5$-point independent set
\[
\{(-1,1,1,0,0),(-1,1,0,1,0), (1,1,-1,0,0),
(1,1,0,-1,0),
(1,1,0,0,-1)\}, 
\]
which gives $X$ is isomorphic to a subset of $Z_2$. 
\end{proof}
\begin{proposition}\label{prop:k3}
Let $X\subset L_{132}$ with $D(X)<D(L_{132})$. Then we have
\[
|X| \leq \mathfrak{M}_3=22. 
\]
If equality holds, then $X=X_3$, $Y_3$, or $Z_3$, up to isomorphism. 
\end{proposition}
\begin{proof}
If $n_i(X,0)=12$ for some $i$, then 
we have $|X|\leq 22$, and the set attaining this bound is $X_3$, $Y_3$, or
$Z_3$ by Lemma~\ref{lem:induc}. 

Suppose $n_i(X,0)\leq 11$ for any $i$. 
If $|X|>22$, then the average of $n_i(X,0)$ is greater than $11$, which gives a contradiction. Therefore $|X|\leq 22$, and if $|X|=22$, then the average of $n_i(X,0)$ is $11$, and 
$n_i(X,0)=11$ for any $i$.  
By Proposition~\ref{prop:k2}, $X$ may contain 
\[
V_3'=\{(0,v) \in L_{132} \mid v \in V_2 \}, 
\]
\[
W_3'=\{(0,w) \in L_{132} \mid w \in W_2 \},
\]
or an $11$-point set obtained by removing a point from 
$X_3'$, $Y_3'$, or $Z_3'$.  

(i) Suppose $X$ contains an $11$-point subset of $X_3'$. 
By considering the automorphism group of $X_3'$, 
$X$ may contain the set in $X_3'$ obtained by removing $(0,-1,0,0,1,1)$, $(0,-1,1,1,0,0)$, $(0,0,-1,0,1,1)$,
or $(0,0,0,-1,1,1)$.
If $X$ contains the set  
$X_3'$ with
$(0,-1,0,0,1,1)$, $(0,-1,1,1,0,0)$, or $(0,0,-1,0,1,1)$ removed, then the set of other candidates of $X$ is still $S_3(1)\cup U_3(3)$, and $|X|<22$.  
Suppose $X$ contains the set 
$X_3'$ with
$(0,0,0,-1,1,1)$ removed. 
Then new candidates of vectors of $X$ are only $(1,0,0,1,-1,0)$ and $(1,0,0,1,0,-1)$,
and $X$ may contain $(1,0,0,1,-1,0)$. 
The three vectors $(-1,1,0,0,1,0)$, $(-1,0,1,0,1,0)$, and $(-1,0,0,0,1,1)$, which are at distance $\sqrt{10}$ from $(1,0,0,1,-1,0)$, are not contained in $X$. 
Therefore 
by 
$|X|=22$, 
the other new candidate $(1,0,0,1,0,-1)$, and two isolated vectors $(-1,0,0,1,0,1)$, and $(-1,0,0,1,1,0)$ 
must be contained in $X$.  
Moreover a $7$-point independent set must be obtained from Matching (i). 
Since $(-1,1,0,0,1,0)$ and $(-1,0,1,0,1,0)$ are not contained in $X$, 
by Lemma~\ref{lem:match}, $(1,-1,0,0,0,1)$ and $(1,0,-1,0,0,1)$ must be contained in $X$, and consequently any element of $U_2(2)$ is contained in $X$. 
This implies $n_2(X,1)=0$, and $X$ is isomorphic to $X_3$, $Y_3$, or $Z_3$. 

(ii) Suppose $X$ contains an $11$-point subset of $Y_3'$. 
By considering the automorphism group of $Y_3'$, 
$X$ may contain the set in $Y_3'$ obtained by removing $(0,-1,0,0,1,1)$, 
$(0,-1,1,1,0,0)$, or $(0,0,1,-1,0,1)$. 
If $X$ contains the set 
$Y_3'$
with
$(0,-1,0,0,1,1)$, or $(0,-1,1,1,0,0)$ removed,
 then 
the set of other candidates of $X$ is still $S_3(1)\cup U_3(2) \cup \mathcal{S}_1$, and $|X|<22$. 
Suppose $X$ contains the set  
$Y_3'$
with $(0,0,1,-1,0,1)$ removed. 
Then a new candidate of an element of $X$ is only $(1,0,0,1,0,-1)$,
and $X$ may contain $(1,0,0,1,0,-1)$.
The three vectors $(-1,1,0,0,0,1)$, $(-1,0,1,0,0,1)$, and 
$(-1,0,0,0,1,1)$, which are at distance $\sqrt{10}$ from 
$(1,0,0,1,0,-1)$, are not contained in $X$. 
By considering Matching (ii), 
we can show $|X|<22$. 

(iii) Suppose $X$ contains an $11$-point subset of $Z_3'$. 
By considering the automorphism group of $Z_3'$, 
$X$ may contain the set in $Z_3'$ obtained by removing $(0,1,-1,0,0,1)$.
Then a new candidate of an element of $X$ is only $(1,0,1,0,0,-1)$,
and $X$ may contain $(1,0,1,0,0,-1)$. 
The three vectors $(-1,1,0,0,0,1)$, $(-1,0,0,1,0,1)$, and $(-1,0,0,0,1,1)$,  which are at distance $\sqrt{10}$ from $(1,0,1,0,0,-1)$, are not contained in $X$. 
By considering Matching (iii),
we can show $|X|<22$.

(iv) Suppose $X$ contains $V_3'$. 
The set of other candidates of $X$ is $S_3(1)\cup U_3(3)\setminus \{(1,-1,1,0,0,0)\}$, 
and the maximum independent set is of order at most $10$ by Matching (i). Thus $|X|<22$. 

(v) Suppose $X$ contains $W_3'$. 
The set of other candidates of $X$ is $S_3(1)\cup U_3(2) \cup \mathcal{S}_1 \setminus \{(1,-1,0,1,0,0)\}$, 
and the maximum independent set is of order at most $10$ by Matching (ii). Thus $|X|<22$. 

Therefore this proposition follows. 
\end{proof}
Finally we prove Theorem~\ref{thm:main}. 
\begin{proof}[Proof of Theorem~\ref{thm:main}]
By Propositions~\ref{prop:k1}--\ref{prop:k3}, the statement holds for $k=1,2,3$. 
By the inductive hypothesis and Lemma~\ref{lem:M_k}, if $|X| \geq \mathfrak{M}_k$, then there exists $i \in \{1,\ldots,n\}$ such that
$n_i(X,0) = \mathfrak{M}_{k-1}$ for $k\geq 4$. By Lemma~\ref{lem:induc}, this theorem holds for any $k$. 
\end{proof}
\section{Classification of the largest $4$-distance sets which contain $\tilde{J}(n,4)$ }
A finite set $X$ in $\mathbb{R}^d$ is called an
$s$-distance set if the set of Euclidean distances of two distinct vectors in $X$ has size $s$. 
The Johnson graph $J(n,m)=(V,E)$, where   
\begin{align*}
V&=\{\{i_1,\ldots, i_m \} \mid 1 \leq i_1<\cdots < i_m \leq n,i_j \in \mathbb{Z} \}, \\
E&=\{(v,u) \mid |v\cap u|=m-1, v,u \in V\},
\end{align*}
is represented into $\mathbb{R}^{n-1}$ as the $m$-distance set $\tilde{J}(n,m)=L_{0,{n-m},m}$.
Indeed $\tilde{J}(n,m) \subset \mathbb{R}^n$, 
but the summation of all entries of any $x \in \tilde{J}(n,m)$ is $m$, and $\tilde{J}(n,m)$ is on a hyperplane isometric to $\mathbb{R}^{n-1}$. 
Bannai, Sato, and Shigezumi \cite{BSS12} investigated  $m$-distance sets containing 
$\tilde{J}(n,m)$. In their paper, for $m\leq 5$ and any $n$, 
the largest $m$-distance sets containing $\tilde{J}(n,m)$ are classified except 
for $(n,m)=(9,4)$. 
In this section, the case $(n,m)=(9,4)$ will be classified. 

The set 
of 
Euclidean distances 
of two distinct points of $\tilde{J}(9,4)$ is $\{\sqrt{2},\sqrt{4},\sqrt{6}, \sqrt{8}\}$. The set of vectors which can be added to $\tilde{J}(9,4)$ while maintaining $4$-distance is the union of the following sets \cite{BSS12}.
\begin{align*}
&X^{(i)}=\left( \left( \frac{2}{3}\right)^7, \left(- \frac{1}{3} \right)^2 \right)^P, & 
&X^{(ii)}=\left( \left( \frac{2}{3}\right)^8, - \frac{4}{3}  \right)^P, \\
&X^{(iii)}=\left(  \frac{4}{3}, \left( \frac{1}{3} \right)^8 \right)^P, &
&X^{(iv)}=\left( \left( \frac{4}{3}\right)^2, \left( \frac{1}{3} \right)^6, - \frac{2}{3}  \right)^P,
\end{align*} 
 where the exponents inside indicate the number of occurrences of the corresponding numbers, 
and the exponent $P$ outside indicates that we should take every permutation. 
They conjectured that  
$\tilde{J}(9,4) \cup X^{(i)}\cup X^{(iii)} \cup \{ (-4/3,(2/3)^8 )\}\cup X^{(iv)'}$ 
is largest, where $(-4/3,(2/3)^8 ) \in X^{(ii)}$, and 
\begin{align*}
X^{(iv)'}=&\left\{(x_1,\ldots,x_9) \in X^{(iv)} \mid x_i=-\frac{2}{3}, 
x_{j_1}=\frac{4}{3},x_{j_2}=\frac{4}{3}, i < j_1,j_2 \right\} \\
&\cup 
\left\{\left( \left(\frac{1}{3} \right)^6,\frac{4}{3},-\frac{2}{3},\frac{4}{3}\right),
\left( \left( \frac{1}{3}\right)^6,\left( \frac{4}{3}\right)^2,-\frac{2}{3}  \right)
 \right\}.
\end{align*}
Actually 
$
X^{(iv)'}
$
is isometric to $X_6$ in Section~\ref{sec:2} by replacing $-2/3$, $1/3$, $4/3$ to $-1$, $0$, $1$, respectively. Let $X^{(iv)''}$ ({\it resp.} $X^{(iv)'''}$) be the set obtained from $Y_6$  ({\it resp.} $Z_6$) by the same manner. 
Using Theorem~\ref{thm:main}, we can classify the largest $4$-distance sets containing $\tilde{J}(9,4)$. 
\begin{theorem}
Let $X \subset \{(x_1,\ldots,x_9) \in \mathbb{R}^9 \mid x_1+ \cdots +x_9=1\}$
be a $4$-distance set which contains $\tilde{J}(9,4)$. Then we have
\[
|X|\leq 258. 
\]
If equality holds, then $X$ is one of the following, up to permutations of coordinates. 
\begin{enumerate}
\item $\tilde{J}(9,4) \cup X^{(i)}\cup X^{(iii)} \cup \{ (-4/3,(2/3)^8 )\}\cup X^{(iv)'}$,
\item $\tilde{J}(9,4) \cup X^{(i)}\cup X^{(iii)} \cup \{ (-4/3,(2/3)^8 )\}\cup X^{(iv)''}$,
\item $\tilde{J}(9,4) \cup X^{(i)}\cup X^{(iii)} \cup \{ (-4/3,(2/3)^8 )\}\cup X^{(iv)'''}$. 
\end{enumerate}
\end{theorem}
\begin{proof}
For any $x \in X^{(i)}\cup X^{(iii)}$, 
$y \in \cup_{j=1}^4 X^{(j)}$, the 
Euclidean distance of 
$x$, $y$ is in $\{\sqrt{2},\sqrt{4},\sqrt{6}, \sqrt{8}\}$, and hence 
$X$ may contain $X^{(i)}\cup X^{(iii)}$. 
The set $X^{(iv)}$ is isometric to 
$L_{162}$ by replacing $-2/3$, $1/3$, $4/3$ to $-1$, $0$, $1$, respectively. 
 Therefore the largest subsets of $X^{(iv)}$ with distances $\{\sqrt{2},\sqrt{4},\sqrt{6}, \sqrt{8}\}$
 are $X^{(iv)'}$, $X^{(iv)''}$,
and $X^{(iv)'''}$, up to permutations of coordinates. 
If $X$ does not contain any element of $X^{(ii)}$, then 
\[
|X|\leq |\tilde{J}(9,4)\cup X^{(i)}\cup X^{(iii)}| +|X^{(iv)'}|= 257.
\] 

If $X$ contains 
$x \in X^{(ii)}$ with $x_i=-4/3$, 
then $X$ cannot contain $y \in X^{(iv)}$ with $y_i=4/3$. 
By re-ordering the vectors, we may assume that the set
\[
X^{(ii)}(t)=\{x \in X^{(ii)} \mid x_i=-4/3,  \exists i \in \{1,\ldots, t\}
\}
\]
is in $X$ for some $t$. Clearly, 
from the definition of $X^{(ii)}(t)$, this set must have size $t$. 
For $t=7,8,9$, $X$ contains at most one  element of $X^{(iv)}$, and hence 
\[|X| \leq |\tilde{J}(9,4)\cup X^{(i)}\cup X^{(iii)}| + t +1\leq 181.
\]
If the set $X^{(ii)}(t)$ is in $X$ for $1\leq t \leq 6$, 
then consider the set of vectors in 
$X \cap X^{(iv)}$ in which the entry $1/3$
occurs in all of the first $t$ positions. 
The final $9-t$ entries of one of these 
vectors forms a vector from $L_{1,6-t,2}$; 
no two vectors in this set can be at the maximum distance. Thus the size of  
\[
|\{x \in X \cap X^{(iv)} \mid x_i=1/3,  \forall i \in \{1,\ldots, t\} \}| 
\]
is bounded by $\mathfrak{M}_{6-t}$. 
It is clear that  
\[
|\{x \in X \cap X^{(iv)} \mid x_i=-2/3, x_{j_1}=4/3, x_{j_2}=4/3, 
\exists i \in \{1,\ldots,t\}, 
\exists j_1,j_2 \in \{t+1,\ldots, 9\}
\}| 
\]
is bounded by $t \binom{9-t}{2}$.
Thus, for $1\leq t \leq 6$, we have
\begin{align*}
|X|& \leq |\tilde{J}(9,4) \cup X^{(i)}\cup X^{(iii)}|+t+\mathfrak{M}_{6-t}+t \binom{9-t}{2}\\
&= \frac{t^3}{3}-\frac{9 t^2}{2}+
\frac{31 t}{6}+257\leq 258,
\end{align*}
and equality holds only if $t=1$. 
The sets attaining this bound are only the three sets in the statement. 
\end{proof}

\section{Remarks on other $M_{mkl}$}
Actually  it is hard to determine $M_{mkl}$ for other $(m, k, l)$ by a similar manner in Section~\ref{sec:2}. 
Fix $m,l$, where $m < l$. By Proposition~\ref{prop:m+k<=l}, if $k\leq l-m$, then $M_{mkl}=\binom{n-1}{m+k-1}\binom{m+k}{m}$. In general there are many largest sets for $k=l-m$. For $k> l-m$, we can inductively construct a large set $X_{k}\subset L_{mkl}$ satisfying $D(X_{k})<D(L_{mkl})$ as follows 
\[
X_{k}=\{(0,x') \mid x' \in X_{k-1}\}\cup \{(x_1,\ldots,x_n) \in L_{mkl} \mid x_1= -1 \}, 
\]
where $X_{l-m}$ is a largest set for $k=l-m$. 
Therefore we have
\[
M_{mkl} \geq 
\mathfrak{M}_{mkl}:=
\binom{m+l-1}{m-1} \binom{k+m+l}{m+l}+\binom{m+l-1}{m}.  
\]
We can generalize Lemma~\ref{lem:M_k} as follows.
\begin{lemma} \label{lem:M_mkl}
Let $X\subset L_{mkl}$ with $D(X)\leq D(L_{mkl})$. 
Suppose $k\geq  m \binom{m+l}{m}-m-l+1$, and $|X| \geq \mathfrak{M}_{mkl}$.  
Then there exists $i \in \{1,\ldots,n\}$ such that
$n_i(X,0)\geq \mathfrak{M}_{m,k-1,l}$. 
\end{lemma}
\begin{proof}
This lemma is immediate because the average of 
$n_i(X,0)$ is  
\[
\frac{1}{n}\sum_{i=1}^{n}n_{i}(X,0)=\frac{k|X|}{m+k+l}
\geq \frac{k\mathfrak{M}_{mkl}}{m+k+l}
=\mathfrak{M}_{m,k-1,l}-\frac{m+l}{m+k+l}\binom{m+k+l}{l}
>\mathfrak{M}_{k-1}-1. \qedhere
\]
\end{proof}
In the manner of Section \ref{sec:2}, it is  hard to classify $M_{mkl}$ for $m-l+1 \leq k \leq m \binom{m+l}{m}-m-l$. 
Moreover it seems to be difficult to give  matchings, like Matching ${\rm (i)}$ or ${\rm (ii)}$, of many possibilities of  $X_k$. 
We need another idea to determine other $M_{mkl}$.

\bigskip

\noindent
\textbf{Acknowledgments.} 
The authors thank Sho Suda for providing useful information. 
The second author is supported by JSPS KAKENHI Grant Numbers 25800011, 26400003.

\end{document}